\documentclass[11pt,a4paper]{article}

\usepackage{epsf,epsfig,amsfonts,amsgen,amsmath,amstext,amsbsy,amsopn,amsthm,lineno
}
\usepackage{ebezier,eepic}
\usepackage{color}

\newtheorem{theorem}{Theorem}

\newtheorem{lemma}{Lemma}
\newtheorem{false statement}{False statement}

\theoremstyle{definition}

\newtheorem{claim}{Claim}
\newtheorem{subclaim}{Claim}[claim]

\newtheorem{problem}{Problem}
\newtheorem{case}{Case}
\newtheorem{subcase}{Subcase}[case]
\newtheorem{subsubcase}{Subsubcase}[subcase]

\newcounter{mathitem}
  {\begin{list}{{$(\roman{mathitem})$}}{
   \setcounter{mathitem}{0}
   \usecounter{mathitem}
   \setlength{\topsep}{0pt plus 2pt minus 0pt}
   \setlength{\parskip}{0pt plus 2pt minus 0pt}
   \setlength{\partopsep}{0pt plus 2pt minus 0pt}
   \setlength{\parsep}{0pt plus 2pt minus 0pt}
   \setlength{\leftmargin}{35pt}
   \setlength{\itemsep}{0pt plus 2pt minus 0pt}}}
  {\end{list}}

\begin{document}
\title{\bf\Large Pairs of Fan-type heavy subgraphs for pancyclicity of 2-connected graphs}

\date{}

\author{Bo Ning\thanks{Email-address: ningbo\_math84@mail.nwpu.edu.cn (B. Ning).}\\[2mm]
\small Department of Applied Mathematics, School of Science,\\
\small Northwestern Polytechnical University, Xi'an, Shaanxi 710129, P.R.~China\\}

\maketitle

\begin{abstract}
A graph $G$ on $n$ vertices is Hamiltonian if it contains a spanning cycle,
and pancyclic if it contains cycles of all lengths from 3 to $n$. In 1984,
Fan presented a degree condition involving every pair of vertices at distance
two for a 2-connected graph to be Hamiltonian. Motivated by Fan's result,
we say that an induced subgraph $H$ of $G$ is $f_1$-heavy if for every
pair of vertices $u,v\in V(H)$, $d_{H}(u,v)=2$ implies
$\max\{d(u),d(v)\}\geq (n+1)/2$. For a given graph $R$, $G$ is called
$R$-$f_1$-heavy if every induced subgraph of $G$ isomorphic to $R$ is
$f_1$-heavy. In this paper we show that for a connected graph $S$ with
$S\neq P_3$ and a 2-connected claw-$f_1$-heavy graph $G$ which is not
a cycle, $G$ being $S$-$f_1$-heavy implies $G$ is pancyclic if $S=P_4,Z_1$
or $Z_2$, where claw is $K_{1,3}$ and $Z_i$ is the path
$a_1a_2a_3\ldots a_{i+2}a_{i+3}$ plus the edge $a_1a_3$.
Our result partially improves a previous theorem
due to Bedrossian on pancyclicity of 2-connected graphs.
\medskip

\smallskip

\end{abstract}

\section{Introduction}
We use Bondy and Murty \cite{Bondy_Murty} for terminology and
notation not defined here and consider simple graphs only.

Let $G$ be a graph and $H$ be a subgraph. Let $x,y$ be two
vertices of $V(H)$. An $(x,y)$-\emph{path} in $H$ is a path
$P$ connecting $x$ and $y$ in $H$.
The \emph{distance} between $x$ and $y$ in $H$, denoted by $d_H(x,y)$, is the
length of a shortest $(x,y)$-path in $H$. When there is no
danger of ambiguity, we use $d(x,y)$ instead of $d_G(x,y)$.

Let $G$ be a graph on $n$ vertices. For a given graph $R$,
$G$ is called $R$-\emph{free} if $G$ contains no induced
subgraph isomorphic to $R$, and \emph{$R$-$f_i$-heavy}
if for every induced subgraph $H$ of $G$ isomorphic to
$R$ and every pair of vertices $u,v\in V(H)$, $d_{H}(u,v)=2$
implies that $\max\{d(u),d(v)\}\geq (n+i)/2$. For a family
$\mathcal{R}$ of graphs, $G$ is called \emph{$\mathcal{R}$-free}
(\emph{$\mathcal{R}$-$f_i$-heavy}) if $G$ is $R$-free
($R$-$f_i$-heavy) for each $R\in\mathcal{R}$. In particular,
similar as in \cite{Ning_Zhang}, we use $R$-$f$-heavy
($\mathcal{R}$-$f$-heavy) instead of $R$-$f_0$-heavy
($\mathcal{R}$-$f_0$-heavy). Note that every
$\mathcal{R}$-free graph is also $\mathcal{R}$-$f_1$-heavy
($\mathcal{R}$-$f$-heavy).

The bipartite graph $K_{1,3}$ is called the \emph{claw}. We
say that its (only) vertex of degree 3 is the \emph{center}
and the other vertices are its \emph{end vertices}.
In this paper, we use the terminology claw-$f_1$-heavy
instead of $K_{1,3}$-$f_1$-heavy.

A graph $G$ on $n$ vertices is said to be \emph{Hamiltonian}
if it contains a \emph{Hamilton cycle}, i.e., a cycle containing
all vertices of $G$, and \emph{pancyclic} if $G$ contains
cycles of all lengths from 3 to $n$. Bedrossian \cite{Bedrossian}
completely characterized all the pairs of forbidden subgraphs
for a 2-connected graph to be Hamiltonian and to be pancyclic.

\begin{theorem}[Bedrossian \cite{Bedrossian}]\label{th1}
Let $R$ and $S$ be connected graphs with $R,S\neq P_{3}$
and let $G$ be a $2$-connected graph. Then $G$ being $\{R,S\}$-free
implies $G$ is Hamiltonian if and only if (up to symmetry)
$R=K_{1,3}$ and $S=P_4,P_5,P_6,C_3,Z_1,Z_2,B,N$ or $W$
(see Figure \ref{fi1}).
\end{theorem}

\begin{theorem}[Bedrossian \cite{Bedrossian}]\label{th2}
Let $R$ and $S$ be connected graphs with $R,S\neq P_3$ and let
$G$ be a $2$-connected graph which is not a cycle. Then $G$ being
$\{R,S\}$-free implies $G$ is pancyclic if and only if
(up to symmetry) $R=K_{1,3}$ and $S=P_4,P_5,Z_1$ or $Z_2$.
\end{theorem}

\begin{figure}
\begin{center}
\begin{picture}(300,180)
\thicklines

\put(5,140){\multiput(20,30)(50,0){5}{\put(0,0){\circle*{6}}}
\put(20,30){\line(1,0){100}} \put(170,30){\line(1,0){50}}
\qbezier[4](120,30)(145,30)(170,30) \put(18,35){$v_1$}
\put(68,35){$v_2$} \put(118,35){$v_3$} \put(168,35){$v_{i-1}$}
\put(218,35){$v_i$} \put(115,10){$P_i$ (Path)}}

\put(265,125){\put(20,30){\circle*{6}} \put(70,30){\circle*{6}}
\put(45,55){\circle*{6}} \put(20,30){\line(1,0){50}}
\put(20,30){\line(1,1){25}} \put(70,30){\line(-1,1){25}}
\put(40,10){$C_3$ (Cycle)}}

\put(0,0){\put(20,30){\circle*{6}} \put(70,30){\circle*{6}}
\multiput(45,55)(0,25){4}{\circle*{6}} \put(20,30){\line(1,0){50}}
\put(20,30){\line(1,1){25}} \put(70,30){\line(-1,1){25}}
\put(45,55){\line(0,1){25}} \put(45,105){\line(0,1){25}}
\qbezier[4](45,80)(45,92.5)(45,105) \put(50,78){$v_1$}
\put(50,103){$v_{i-1}$} \put(50,128){$v_i$} \put(40,10){$Z_i$}}

\put(90,0){\put(45,40){\circle*{6}} \put(45,40){\line(-1,1){25}}
\put(45,40){\line(1,1){25}} \put(20,65){\line(1,0){50}}
\multiput(20,65)(50,0){2}{\multiput(0,0)(0,30){2}{\put(0,0){\circle*{6}}}
\put(0,0){\line(0,1){30}}} \put(25,10){$B$ (Bull)}}

\put(180,0){\multiput(20,30)(50,0){2}{\multiput(0,0)(0,30){2}{\put(0,0){\circle*{6}}}
\put(0,0){\line(0,1){30}}}
\multiput(45,85)(0,30){2}{\put(0,0){\circle*{6}}}
\put(45,85){\line(0,1){30}} \put(20,60){\line(1,0){50}}
\put(20,60){\line(1,1){25}} \put(70,60){\line(-1,1){25}}
\put(25,10){$N$ (Net)}}

\put(270,0){\put(45,30){\circle*{6}} \put(20,55){\line(1,0){50}}
\put(45,30){\line(1,1){25}} \put(45,30){\line(-1,1){25}}
\multiput(20,55)(0,30){2}{\put(0,0){\circle*{6}}}
\multiput(70,55)(0,30){3}{\put(0,0){\circle*{6}}}
\put(20,55){\line(0,1){30}} \put(70,55){\line(0,1){60}}
\put(10,10){$W$ (Wounded)}}

\label{fi1}
\end{picture}

\caption{Graphs $P_i,C_3,Z_i,B,N$ and $W$}
\end{center}
\end{figure}
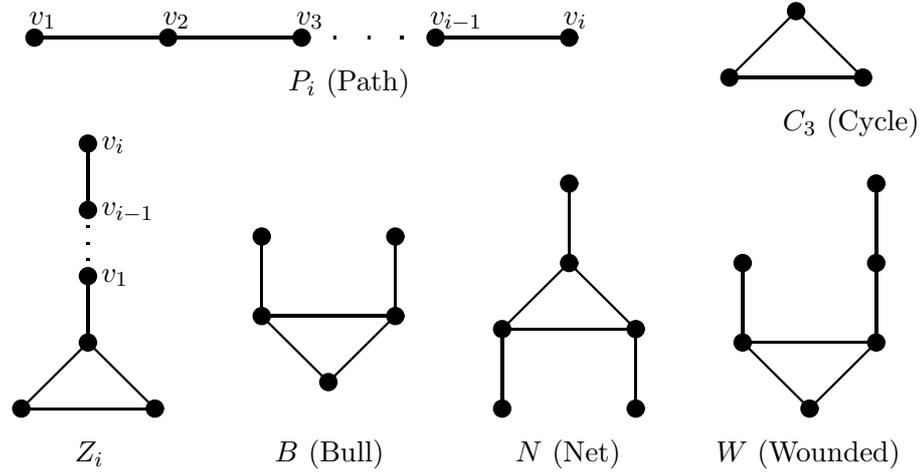
In 1984, Fan \cite{Fan} presented a degree condition
(so-called Fan's condition) involving every pair of
vertices at distance two for a 2-connected graph to
be Hamiltonian.

\begin{theorem}[Fan \cite{Fan}]\label{th3}
Let $G$ be a $2$-connected graph on $n$ vertices.
If $\max\{d(u),d(v)\}\geq n/2$ for every pair of
vertices $u,v$ such that $d(u,v)=2$, then $G$ is
Hamiltonian.
\end{theorem}

Obviously, Fan's condition is equivalent to every
2-connected $P_3$-$f$-heavy graph is Hamiltonian.
By restricting Fan's condition to some induced
subgraphs of 2-connected graphs, Ning and Zhang
\cite{Ning_Zhang} extended Theorem \ref{th1}
as follows.

\begin{theorem}[Ning and Zhang \cite{Ning_Zhang}]\label{th4}
Let $R$ and $S$ be connected graphs with $R,S\neq P_3$
and let $G$ be a $2$-connected graph. Then $G$ being
$\{R,S\}$-$f$-heavy implies $G$ is Hamiltonian if
and only if (up to symmetry) $R=K_{1,3}$ and
$S=P_4,P_5,P_6,Z_1,Z_2,B,N$ or $W$.
\end{theorem}

In this paper, our aim is to find the corresponding
Fan-type heavy subgraph conditions for a 2-connected
graph to be pancyclic. First, from a well known result,
we can deduce that every 2-connected $P_3$-$f_1$-heavy 
graph is pancyclic.

\begin{theorem}[Benhocine and Wojda \cite{Benhocine_Wojda}]\label{th5}
Let $G$ be a $2$-connected graph on $n\geq 3$ vertices. If
$G$ is $P_3$-$f$-heavy, then $G$ is pancyclic unless $n=4r$,
$r>2$, and $G=F_{4r}$ (see Figure $2$), or $n$ is even and
$G=K_{n/2,n/2}$ or else $n\geq 6$ is even and $G=K_{n/2,n/2}-e$.
\end{theorem}

Furthermore, we can see that $P_3$ is the only connected
graph $S$ such that every 2-connected $S$-$f_1$-heavy graph
is pancyclic. For details, see \cite[Theorem 13]{Faudree_Gould}.
So we can pose the following problem naturally.

\begin{problem} Which two connected graphs $R$ and $S$
other than $P_3$ imply that every 2-connected
$\{R,S\}$-$f_1$-heavy graph is pancyclic?
\end{problem}

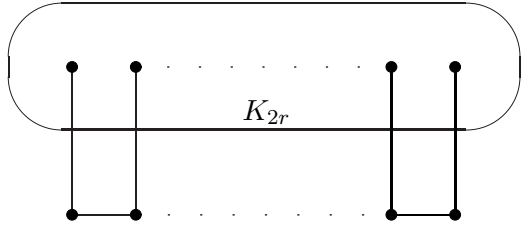
\begin{figure}
\begin{center}
\setlength{\unitlength}{1.4mm}
\begin{picture}(60,40)
\put(30,26){\oval(48,12)}
\put(12,26){\line(0,-1){14}}
\put(18,26){\line(0,-1){14}}
\put(12,12){\line(1,0){6}}
\put(12,26){\circle*{1}}
\put(18,26){\circle*{1}}
\put(12,12){\circle*{1}}
\put(18,12){\circle*{1}}
\put(48,26){\line(0,-1){14}}
\put(42,26){\line(0,-1){14}}
\put(48,12){\line(-1,0){6}}
\put(48,26){\circle*{1}}
\put(42,26){\circle*{1}}
\put(48,12){\circle*{1}}
\put(42,12){\circle*{1}}
\dottedline{3}(18,26)(42,26)
\dottedline{3}(18,12)(42,12)
\put(28,21){$K_{2r}$}

\label{fi2}
\end{picture}

\caption{The Graph $F_{4r}$}
\end{center}
\end{figure}

By Theorem \ref{th2}, we know that $R=K_{1,3}$ (up to symmetry)
and $S$ must be one of $Z_1$, $Z_2$, $P_4$ and $P_5$.

In this paper, we mainly prove the following result.

\begin{theorem}\label{th6}
Let $G$ be a $2$-connected graph which is not a cycle. If $G$ is
$\{K_{1,3},Z_2\}$-$f_1$-heavy, then $G$ is pancyclic.
\end{theorem}

As a corollary of Theorem \ref{th6}, we have
\begin{theorem}\label{th7}
Let $G$ be a $2$-connected graph which is not a cycle. If $G$ is
$\{K_{1,3},P_4\}$-$f_1$-heavy, then $G$ is pancyclic.
\end{theorem}

In \cite{Bedrossian_Chen_Schelp}, Bedrossian et al. proved a
theorem as follows.

\begin{theorem}[Bedrossian, Chen and Schelp \cite{Bedrossian_Chen_Schelp}]\label{th8}
Let $G$ be a $2$-connected graph on $n$ vertices. If $G$ is
$\{K_{1,3},Z_1\}$-$f$-heavy, then $G$ is pancyclic unless
$G=F_{4r}$ or $G=K_{n/2,n/2}$ or $G=K_{n/2,n/2}-e$ or
else $G$ is a cycle.
\end{theorem}
By Theorem \ref{th8}, we have

\begin{theorem}\label{th9}
Let $G$ be a $2$-connected graph which is not a cycle.
If $G$ is $\{K_{1,3},Z_1\}$-$f_1$-heavy, then $G$ is
pancyclic.
\end{theorem}

Combining with Theorems \ref{th6}, \ref{th7} and \ref{th9},
we obtain Theorem \ref{th10}, which partially answers
Problem 1.

\begin{theorem}\label{th10}
Let $S$ be a connected graph with $S\neq P_3$ and let
$G$ be a $2$-connected claw-$f_1$-heavy graph which is
not a cycle. Then $G$ being $S$-$f_1$-heavy implies $G$
is pancyclic if $S=P_4,Z_1$ or $Z_2$.
\end{theorem}

The rest of this paper is organized as follows.
In Section \ref{se2}, we will give additional terminology
and list some useful lemmas. The proof of Theorem \ref{th6}
will be postponed to Section \ref{se3}.
\section{Preliminaries}\label{se2}
In this section, we first introduce some additional terminology
and notation and then present four lemmas which will be used
in our proof of Theorem \ref{th6}.

Let $G$ be a graph and $S$ be a subset of of $V(G)$. We use
$G[S]$ to denote the subgraph of $G$ induced by $S$ and $G-S$
to denote $G[V(G)\setminus S]$. In particular, if $S=\{u\}$,
then we use $G-u$ instead of $G-\{u\}$. If $S=\{x_i:1\leq i\leq 5\}$
and $G[S]$ is isomorphic to $Z_2$, then we say that
$\{x_1,x_2,x_3;x_4,x_5\}$ induces a $Z_2$, where $x_1x_2x_3x_1$
is a triangle and $x_1$ is the vertex of degree 3 in $G[S]$.
If $S=\{x_i:1\leq i\leq 4\}$ and $G[S]$ is isomorphic to $K_{1,3}$,
then we say that $\{x_1;x_2,x_3,x_4\}$ induces a claw, where $x_1$
is the center, and $x_2,x_3,x_4$ are the end vertices.

Let $k,l$ $(k<l)$ be two integers. We say that $G$ contains a
\emph{$k$-cycle} if $G$ contains a cycle of length $k$, and $G$
contains \emph{$[k,l]$-cycles} if $G$ contains cycles of all
lengths from $k$ to $l$. In particular, for a vertex $u\in V(G)$,
we say that $G$ contains a \emph{$u$-triangle} if $G$ contains the
cycle $uxyu$,
where $x,y\in V(G)$.

A vertex $v$ of a graph $G$ on $n$ vertices is called \emph{heavy}
if $d(v)\geq n/2$, and \emph{super-heavy} if $d(v)\geq (n+1)/2$.
For two vertices $u,v$ of $G$, $\{u,v\}$ is called a \emph{heavy-pair}
if $d(u)+d(v)\geq n$ and a \emph{super-heavy pair} if $d(u)+d(v)\geq n+1$.

\begin{lemma}[Benhocine and Wojda \cite{Benhocine_Wojda}] \label{le1}
Let $G$ be a graph on $n\geq 4$ vertices and $C$ be a cycle of
length $n-1$ in $G$. If $d(x)\geq n/2$ for the vertex $x\in V(G)\backslash V(C)$,
then $G$ is pancyclic.
\end{lemma}

\begin{lemma}[Bondy \cite{Bondy}]\label{le2}
Let $G$ be a graph on $n$ vertices with a Hamilton cycle $C$.
If there exist two vertices $x,y\in V(G)$ such that $d_C(x,y)=1$
and $d(x)+d(y)\geq n+1$, then $G$ is pancyclic.
\end{lemma}

\begin{lemma}[Hakimi and Schmeichel \cite{Schmeichel_Hakimi}]\label{le3}
Let $G$ be a graph on $n$ vertices with a Hamilton cycle $C$.
If there exist two vertices $x,y\in V(G)$ such that $d_C(x,y)=1$
and $d(x)+d(y)\geq n$, then $G$ is pancyclic unless $G$ is
bipartite or else $G$ is missing only an $(n-1)$-cycle.
\end{lemma}

\begin{lemma}[Ferrara, Jacobson and Harris \cite{Ferrara_Jacobson_Harris}]\label{le4}
Let $G$ be a graph on $n$ vertices with a Hamilton cycle $C$.
If there exist two vertices $x,y\in V(G)$ such that $d_C(x,y)=2$
and $d(x)+d(y)\geq n+1$, then $G$ is pancyclic.
\end{lemma}
\section{Proof of Theorem \ref{th6}}\label{se3}

We prove Theorem \ref{th6} by contradiction. Suppose that $G$
satisfies the condition of Theorem \ref{th6} but is not pancyclic.
Since the result is easy to verify for $3\leq n\leq 5$, we assume
that $n\geq 6$.

If $G$ is $\{K_{1,3},Z_2\}$-free, then by Theorem \ref{th2},
$G$ is pancyclic. Thus we assume that $G$ contains an induced
claw or an induced $Z_2$. Therefore, there is a super-heavy
vertex, say $u\in V(G)$. Set $G'=G-u$. Since $G$ is $\{K_{1,3},Z_2\}$-$f_1$-heavy,
$G'$ is $\{K_{1,3},Z_2\}$-$f$-heavy. If $G'$ is $2$-connected, then by
Theorem \ref{th4}, $G'$ is Hamiltonian. Hence $G$ is pancyclic by Lemma
\ref{le1}. Now, it will be assumed that $G'$ is not 2-connected. Then there
exists a vertex $v\in V(G)$ ($v\neq u$) such that $G-\{u,v\}$ is not
connected. By Theorem \ref{th4}, $G$ is Hamiltonian. Hence $G-\{u,v\}$
consists of two components $H_1$ and $H_2$. Without loss of generality,
we assume that $|V(H_1)|\leq |V(H_2)|$, where $V(H_1)=\{x_1,x_2,\ldots,x_{h_1}\}$
and $V(H_2)=\{y_1,y_2,\ldots,y_{h_2}\}$.
Let $C=uy_1\cdots y_{h_2}vx_{h_1}\cdots x_1u$ be a
Hamilton cycle with the given orientation. In the following,
for any two vertices $w_1,w_2\in V(C)$, we use $C[w_1,w_2]$ to
denote the segment of $C$ from $w_1$ to $w_2$ along the
orientation. Set $G_1=G[V(H_1)\cup \{u\}]$ and $G_2=G[V(H_2)\cup \{u\}]$.

\begin{claim}\label{cl1}
There are no super-heavy vertices in $H_1$.
\end{claim}
\begin{proof}
For any vertex $x\in V(H_1)$, $x$ is adjacent to at most $u,v$ and
all the vertices in $H_1$ except for itself. Therefore,
$d(x)\leq d_{H_1}(x)+2\leq h_1-1+2\leq n/2$. Hence $H_1$
contains no super-heavy vertices.
\end{proof}

\begin{claim}\label{cl2}
$N_{G_2}(u)\backslash \{y_1\}\subseteq N(y_1)$.
\end{claim}

\begin{proof}
If there exists a vertex $y_i\in N_{G_2}(u)\backslash \{y_1\}$
such that $y_iy_1\notin E(G)$, then $\{u;x_1,y_1,y_i\}$ induces
a claw. By Claim \ref{cl1}, $x_1$ is not super-heavy. Since $G$ is
claw-$f_1$-heavy, $y_1$ is super-heavy. Hence $\{u,y_1\}$ is
a super-heavy pair such that $d_{C}(u,y_1)=1$. By Lemma \ref{le2},
$G$ is pancyclic.
\end{proof}

\begin{claim}\label{cl3}
There are no super-heavy pairs with distance one or two along
the orientation of a Hamilton cycle in $G$.
\end{claim}

\begin{proof}
Suppose not. By Lemma \ref{le2} or \ref{le4}, $G$ is pancyclic.
\end{proof}

\begin{case}
$h_1=1$.
\end{case}

\begin{subcase}
$uv\in E(G)$.
\end{subcase}
Note that $G$ cannot be bipartite or missing an $(n-1)$ cycle,
so if Lemma \ref{le3} applies to $G$ then $G$ is pancyclic. If $u$ is
adjacent to every vertex in $C$, then $G$ is pancyclic. Now we
can choose a vertex $y_i\in N_{G_2}(u)$ such that $uy_{i+1}\notin E(G)$.
Let $y_j$ be the first vertex on $C[y_i,y_{h_2}]$ such that
$uy_{j+1}\in E(G)$, where assume that $y_{h_2+1}=v$.
Obviously, $j\geq i+1$.
\begin{claim}\label{cl4}
$i\geq 2$.
\end{claim}

\begin{proof}
Assume there exists $y\in V(H_2)$ such that $y_1y\in E(G)$
and $uy\notin E(G)$. By Claim \ref{cl2}, we have
$N_{G_2}(u)\backslash \{y_1\}\subset N(y_1)$.
Since $d(u)\geq (n+1)/2$ and $u,y\in N(y_1)\backslash N(u)$,
$d(y_1)\geq d(u)-3+2\geq (n-1)/2$. Therefore, $\{u,y_1\}$ is a
heavy-pair such that $d_{C}(u,y_1)=1$. By Lemma \ref{le3},
$G$ is pancyclic. Also, since $y_1y_2\in E(G)$, then
$uy_2\in E(G)$ and $i\geq 2$.
\end{proof}
Next we assume that $i\leq h_2-2$. Note that
$y_i,y_{i+1},y_{i+2}\in C[y_2,y_{h_2}]$.

\begin{claim}\label{cl5}
$j\geq i+2$.
\end{claim}
\begin{proof}
Assume that $j=i+1$. First, we have
$uy_i,uy_{i+2}\in E(G)$ and $uy_{i+1}\notin E(G)$.

If $y_iy_{i+2}\notin E(G)$, then $\{u;x_1,y_i,y_{i+2}\}$
induces a claw. Since $d(x_1)=2<(n+1)/2$ and $G$ is
claw-$f_1$-heavy, $\{y_i,y_{i+2}\}$ is a super-heavy
pair such that $d_{C}(y_i,y_{i+2})=2$, which contradicts
to Claim \ref{cl3}.

Now assume that $y_iy_{i+2}\in E(G)$. If $y_1y_{i+1}\in E(G)$,
then it follows $d(y_1)\geq (n-1)/2$ from Claim \ref{cl2}.
Hence $\{u,y_1\}$ is a heavy pair with $d_{C}(u,y_1)=1$,
and $G$ is pancyclic by Lemma \ref{le3}. Therefore,
$y_1y_{i+1}\notin E(G)$. We set $G'=G-y_i$. Clearly,
$C'=C[y_{i+2},y_i]y_iy_{i+2}$ is a Hamilton cycle in $G'$. Moreover,
$u,y_1$ satisfy that $d_{G'}(u)+d_{G'}(y_1)=d(u)+d(y_1)\geq (n+1)/2+(n-3)/2=n-1$
and $d_{C'}(u,y_1)=1$. By Lemma \ref{le3}, $G'$ is
pancyclic. Together with the cycle $C$, $G$ is pancyclic.
\end{proof}
By Claim \ref{cl5}, we obtain $uy_{i+2}\notin E(G)$.

\begin{claim}\label{cl6}
$vy_{i+1}\in E(G)$.
\end{claim}

\begin{proof}
Assume that $vy_{i+1}\notin E(G)$.

\begin{subclaim}\label{su1}
$vy_{i+2}\notin E(G)$.
\end{subclaim}

\begin{proof}
Assume that $vy_{i+2}\in E(G)$. Then $\{v,x_1,u;y_{i+2},y_{i+1}\}$
induces a $Z_2$. If $v$ is a super-heavy vertex, then
$\{u,v\}$ is a super-heavy pair such that $d_{C}(u,v)=2$,
contradicting to Claim \ref{cl3}. Now assume that $v$ is
not super-heavy. Note that $x_1$ is not super-heavy.
Since $G$ is $Z_2$-$f_1$-heavy, $\{y_{i+1},y_{i+2}\}$
is a super-heavy pair such that $d_C(y_{i},y_{i+1})=1$,
contradicting to Claim \ref{cl3}.
\end{proof}

\begin{subclaim}\label{su2}
$vy_{i}\notin E(G)$.
\end{subclaim}

\begin{proof}
Assume that $vy_{i}\in E(G)$. By Claim \ref{su1}, we
have $vy_{i+2}\notin E(G)$. Note that $vy_{i+1}\notin E(G)$
by the initial hypothesis. If $y_iy_{i+2}\notin E(G)$,
then $\{y_i,u,v;y_{i+1},y_{i+2}\}$ induces a $Z_2$.
Since $v$ is not super-heavy, $y_{i+1}$ is super-heavy.
Hence either $\{y_{i+1},y_{i+2}\}$ or  $\{y_{i+1},y_i\}$
is a super-heavy pair, a contradiction by Claim \ref{cl3}.
If $y_iy_{i+2}\in E(G)$, then $\{y_i,y_{i+1},y_{i+2};v,x_1\}$
induces a $Z_2$. Since $v$ is not super-heavy,
$\{y_{i+1}, y_{i+2}\}$ is a super-heavy pair
such that $d_C(y_i,y_{i+1})=1$, a contradiction
by Claim \ref{cl3}.
\end{proof}

\begin{subclaim}\label{su3}
$y_i$ is super-heavy.
\end{subclaim}

\begin{proof}
By Claims \ref{su2} and the initial hypothesis, we have $vy_i\notin E(G)$
and $vy_{i+1}\notin E(G)$. Since $\{u,v,x_1;y_i,y_{i+1}\}$ induces a
$Z_2$ and $x_1$ is not super-heavy, $y_i$ is super-heavy.
\end{proof}
By Claim \ref{cl4}, we have $i\geq 2$, and this implies $y_{i-1}$
is well-defined.
\begin{subclaim}\label{su4}
$y_{i-1}y_{i+1}\notin E(G)$, $uy_{i-1}\in E(G)$, $y_iy_{i+2}\notin E(G)$
and $y_{i-1}y_{i+2}\notin E(G)$.
\end{subclaim}

\begin{proof}
By Claim \ref{su3}, $y_i$ is super-heavy. If $y_{i-1}y_{i+1}\in E(G)$,
then $G$ is pancyclic by Lemma \ref{le1}.

If $uy_{i-1}\notin E(G)$, then $\{y_i;y_{i-1},y_{i+1},u\}$ induces
a claw. Hence either $y_{i-1}$ or $y_{i+1}$ is super-heavy. Therefore,
either $\{y_{i-1},y_i\}$ or $\{y_i,y_{i+1}\}$ is a super-heavy
pair such that $d_C(y_{i-1},y_i)=d_C(y_i,y_{i+1})=1$, a
contradiction by Claim \ref{cl3}.

By Claim \ref{cl2} and Lemma \ref{le3}, $y_1y_{i+1}\notin E(G)$.
If $y_iy_{i+2}\in E(G)$, then set $G'=G-y_{i+1}$. Now
$C'=vx_1uy_1\ldots y_iy_{i+2}\ldots y_{h_2}v$ is a Hamilton cycle
in $G'$, and $d_{G'}(u)+d_{G'}(y_1)\geq n-1=|G'|$ by Claim 2.
By Lemma \ref{le3}, $G'$ is either pancyclic, bipartite, or
missing only an $(n-2)$-cycle.  Since
$C'=vx_1uy_1\cdots y_iy_{i+2}\cdots y_{h_2}v$ is an $(n-1)$-cycle
and $C''=vuy_1\cdots y_iy_{i+2}\cdots y_{h_2}v$ is an $(n-2)$-cycle
in $G'$, $G'$ is pancyclic. Therefore, $G$ is pancyclic.

If $y_{i-1}y_{i+2}\in E(G)$, then set $G'=G-y_{i+1}$. Now
$C'=uy_1\ldots y_{i-1}y_{i+2}\ldots y_{h_2}vx_1u$ is a Hamilton
cycle in $G''=G'-y_i$ and $d_{G'}(y_i)\geq(n-1)/2=|G'|/2$.
By Lemma \ref{le1}, $G'$ is pancyclic. Together with the cycle
$C$, $G$ is pancyclic.
\end{proof}
By Claim \ref{su4}, $\{y_i,u,y_{i-1};y_{i+1},y_{i+2}\}$
induces a $Z_2$. Since $G$ is $Z_2$-$f_1$-heavy, either
$y_{i-1}$ or $y_{i+1}$ is super-heavy. Then either
$\{y_{i-1},y_i\}$ or $\{y_{i+1},y_i\}$ is a super-heavy
pair such that $d_C(y_{i-1},y_i)=d_C(y_{i+1},y_i)=1$.
By Claim \ref{cl3}, a contradiction.
\end{proof}

\begin{claim}\label{cl7}
For any $k\in \{i+1,\cdots,j\}$, $vy_k\in E(G)$.
\end{claim}

\begin{proof}
By Claim \ref{cl6}, we have $vy_{i+1}\in E(G)$. Now we
show that $vy_k\in E(G)$ for any $k\in \{i+2,\cdots,j\}$.
Otherwise, assume that $y_t$ is the first vertex on
$C[y_{i+2},y_j]$ such that $vy_t\notin E(G)$. Note that
for any $k\in \{i+1,\cdots,j\}$, $uy_k\notin E(G)$.
We have $uy_{t-1},uy_t\notin E(G)$. Then $\{v,x_1,u;y_{t-1},y_t\}$
induces a $Z_2$. Since $x_1,v$ are not super-heavy,
$\{y_{t-1},y_t\}$ is a super-heavy pair such that
$d_C(y_{t-1},y_t)=1$. By Claim \ref{cl3}, a
contradiction, hence $vy_k\in E(G)$.
\end{proof}

Note that since $j\geq i+2$ and $i$ could be selected
to be $\leq h_2-2$, then if $(j+1)\leq h_2-2$, let $i=j+1$
and repeat the previous arguments to conclude that for any
vertex $y\in \{y_2,y_3,\cdots,y_{h_2-2}\}$ such that
$uy\notin E(G)$, we have $vy\in E(G)$. Hence
$d_{C[y_1,y_{h_2-2}]}(u)+d_{C[y_1,y_{h_2-2}]}(v)\geq h_2-2$.
If $uy_{h_2-1}\in E(G)$ or $vy_{h_2-1}\in E(G)$ or
$uy_{h_2}\in E(G)$, then $d_{G_2}(u)+d_{G_2}(v)\geq h_2=n-3$.
This implies that $d(u)+d(v)\geq n+1$. By Claim \ref{cl3},
a contradiction. Otherwise, assume that
$uy_{h_2-1},uy_{h_2}\notin E(G)$ and $vy_{h_2-1}\notin E(G)$.
Then $\{v,x_1,u;y_{h_2},y_{h_2-1}\}$ induces a $Z_2$.
It follows that $\{y_{h_2},y_{h_2-1}\}$ is a super-heavy pair
such that $d_C(y_{h_2-1},y_{h_2})=1$, contradicting to Claim \ref{cl3}.

\begin{subcase}
$uv\notin E(G)$.
\end{subcase}
By Claim \ref{cl2}, $N_{G_2}(u)\backslash \{y_1\}\subseteq N(y_1)$.
If $uy_2\notin E(G)$, then since $u$ is super-heavy and $u,y_2\in N(y_1)\backslash N(u)$,
$y_1$ is super-heavy. Hence $\{u,y_1\}$ is a super-heavy pair such that
$d_C(u,y_1)=1$, a contradiction by Claim \ref{cl3}. If $uy_2\in E(G)$,
then we have $d(y_1)\geq (n-1)/2$ and $\{u,y_1\}$ is a heavy-pair
such that $d_C(u,y_1)=1$. By Lemma \ref{le3}, $G$ is either pancyclic,
bipartite, or missing only an $(n-1)$-cycle. The cycle $uy_1y_2u$
(a triangle) is odd, so $G$ is not bipartite. Since
$C'=ux_1vy_{h_2},\ldots,y_2u$ is an $(n-1)$-cycle, $G$ is
pancyclic.
\begin{case}
$h_1\geq 2$.
\end{case}

\begin{subcase}
$G_1$ contains a $u$-triangle.
\end{subcase}

Without loss of generality, we denote a $u$-triangle in
$G_1$ by $ux_{k}x_{k'}u$ where $k<k'$.

\begin{subsubcase}
$u$ is not adjacent to every vertex of $H_2$.
\end{subsubcase}
Let $y_i\in V(H_2)$ be the vertex such that $uy_i\notin E(G)$
and $i$ is as small as possible. Note that $\{u,x_k,x_{k'};y_{i-1},y_i\}$
induces a $Z_2$. By Claim \ref{cl1}, $y_{i-1}$ is super heavy.
So if $i=2$ then $\{u,y_1\}$ is a super-heavy pair such that
$d_{C}(u,y_1)=1$, a contradiction by Claim \ref{cl3}.
Therefore $i\geq 3$ and $uy_2\in E(G)$.

If there exists $t\in \{1,2,\ldots,h_1-1\}$ such that
$ux_t\in E(G)$ and $ux_{t+1}\notin E(G)$, then
$\{u,y_1,y_2;x_t,x_{t+1}\}$ induces a $Z_2$. Note
that $x_t$ is not super-heavy. Since $G$ is $Z_2$-$f_1$-heavy,
$y_1$ is super-heavy. Hence $\{u,y_1\}$ is a super-heavy pair
such that $d_C(u,y_1)=1$, contradicting to Claim \ref{cl3}.
Therefore, $u$ is adjacent to every vertex of $H_1$. Note that
$C'=ux_1\cdots x_iu$ is an $(i+1)$-cycle, where $2\leq i\leq h_1$,
and $G$ contains $[3,h_1+1]$-cycles. If $i=h_2$, then $u$ is
adjacent to every vertex of $H_2$ other than $y_{h_2}$.
It follows $G$ contains $[h_1+4,n]$-cycles. Furthermore,
$C'=ux_2\cdots x_{h_1}vy_{h_2}y_{h_2-1}u$ is an $(h_1+3)$-cycle.
If $h_1\geq 3$, then $C'=ux_3\cdots x_{h_1}vy_{h_2}y_{h_2-1}u$
is an $(h_1+2)$-cycle, and $G$ is pancyclic. If $h_1=2$ and
$h_2\geq 4$, then we can easily find a 4-cycle in $G$, and
$G$ is pancyclic. If $h_1=2$ and $h_2=2$ or 3, then $n=6$ or 7.
In these two cases, the result is easy to verify.

Now we suppose that $3\leq i\leq h_2-1$ and try to get a contradiction.
If there exists $y_k\in N_{G_2}(u)$ such that $y_ky_{i-2}\notin E(G)$
and $y_k\neq y_{i-2}$, then $\{u;x_1,y_k,y_{i-2}\}$ induces a claw.
Since $G$ is claw-$f_1$-heavy and $x_1$ is not super-heavy, $y_{i-2}$
is super-heavy. Therefore, $\{y_{i-2},y_{i-1}\}$ is a super-heavy pair
such that $d_C(y_{i-2},y_{i-1})=1$, a contradiction by Claim \ref{cl3}.
So, $N_{G_2}(u)\backslash \{y_{i-2}\}\subseteq N(y_{i-2})$.

If $uv\in E(G)$, then we set $G'=G-V(H_1)$. Since
$N(u)\cup \{u\}\backslash (V(H_1)\cup\{v,y_{i-2}\})\subseteq N(y_{i-2})$,
we have $d(y_{i-2})\geq d(u)+1-h_1-2\geq (n+1)/2-h_1-1$. Furthermore,
we obtain $d_{G'}(y_{i-2})+d_{G'}(y_{i-1})=d(y_{i-2})+d(y_{i-1})\geq n-h_1=|G'|$.
Let $C'=uvy_{h_2}\cdots y_1u$. Then $C'$ is a Hamilton cycle in $G'$ and
$d_{C'}(y_{i-2},y_{i-1})=1$. By Lemma \ref{le3}, $G'$ is either pancyclic,
bipartite, or missing only a $(|G'|-1)$-cycle. But $G'$ contains the
triangle $uy_1y_2u$, hence it is not bipartite. Note that $G$ contains
the cycle $C''=uvy_{h_2}\cdots y_2u$ of length $|G'|-1$. Hence $G'$ is
pancyclic, and this implies that $G$ contains $[3,|G'|]$-cycles.
Since $u$ is adjacent to every vertex of $H_1$, $G$ contains
$[|G'|+1,n]$-cycles. Hence $G$ is pancyclic.

If $uv\notin E(G)$, then we set $G'=G-(V(H_1)\backslash \{x_{h_1}\})$.
Now we have $d(y_{i-2})\geq d(u)-h_1-1+1\geq (n+1)/2-h_1$.
And we obtain $d_{G'}(y_{i-2})+d_{G'}(y_{i-1})\geq d(y_{i-2})+d(y_{i-1})\geq n+1-h_1=|G'|$.
Similarly, we can prove that $G$ is pancyclic.

\begin{subsubcase}
$u$ is adjacent to every vertex of $H_2$.
\end{subsubcase}
Note that $uy_1y_2u$ is a $u$-triangle. If there exists
a vertex $x_t\in V(H_1)$ such that $ux_t\in E(G)$ and $ux_{t+1}\notin E(G)$,
then $\{u,y_1,y_2;x_t,x_{t+1}\}$ induces a $Z_2$. This implies that
$y_1$ is super-heavy. Hence $\{u,y_1\}$ is a super-heavy pair such
that $d_C(u,y_1)=1$, a contradiction by Claim \ref{cl3}. If $u$
is adjacent to every vertex in $H_1$, then $u$ is adjacent to
all vertices of $V(G)\backslash \{u,v\}$. This implies that
$d(u)\geq n-2$, and $d(u)+d(y_1)\geq n$. By Lemma \ref{le3},
$G$ is either pancyclic, bipartite, or missing only an $(n-1)$-cycle.
Since $u$ is adjacent to every vertex in $H_2$, $G$ is neither
bipartite nor missing $(n-1)$-cycles. It follows that $G$ is
pancyclic.

\begin{subcase}
$G_1$ contains no $u$-triangles.
\end{subcase}

We first show that $N_{G_1}(u)=\{x_1\}$. Suppose not. If there
is a vertex $x\in N_{G_1}(u)$ such that $x\neq x_1$, then since
$G_1$ contains no $u$-triangles, we have $xx_1\notin E(G)$.
Now $\{u;x,x_1,y_1\}$ induces a claw. It follows that either
$x$ or $x_1$ is super-heavy, which contradicts to Claim \ref{cl1}.

If there exist two consecutive vertices, say $y_i,y_{i+1}\in V(H_2)$,
such that $uy_i,uy_{i+1}\in E(G)$, then $\{u,y_i,y_{i+1};x_1,x_2\}$
induce a $Z_2$. Hence $\{y_i,y_{i+1}\}$ is a super-heavy pair such
that $d_C(y_i,y_{i+1})=1$, a contradiction by Claim \ref{cl3}.

Therefore for any $y_i\in V(H_2)\setminus\{y_{h_2}\}$, $|\{uy_i,uy_{i+1}\}\cap E(G)|\leq 1$.
This implies that $u$ is adjacent to only one vertex $x_1$ in $H_1$
and at most $(h_1+1)/2$ vertices in $H_2$ and maybe adjacent to $v$
or not. Hence we have $(n+1)/2\leq d(u)\leq 1+1+(h_2+1)/2$. This
implies that $h_2\geq n-4$. Noting that $h_2=n-2-h_1\leq n-2-2=n-4$,
we have $h_2=n-4,h_1=2$, $uv\in E(G)$ and $N_{G_2}(u)=\{y_{2k+1}:k=0,1,\ldots,(n-5)/2\}$,
where $n$ is odd.

If $y_1y_3\notin E(G)$, then $\{u;x_1,y_1,y_3\}$ induces a claw.
Since $G$ is claw-$f_1$-heavy, $\{y_1,y_3\}$ is a super-heavy pair
such that $d_C(y_1,y_3)=2$. By Claim \ref{cl3}, a contradiction.

If $y_1y_3\in E(G)$, then $\{u,y_1,y_3;x_1,x_2\}$ induces a $Z_2$.
Since $G$ is $Z_2$-$f_1$-heavy, $\{y_1,y_3\}$ is a super-heavy pair
such that $d_{C}(y_1,y_3)=2$. By Claim \ref{cl3}, also a contradiction.

The proof is complete. {\hfill$\Box$}
\section*{Acknowledgements}
This work was supported by NSFC (No.~11271300) and the Doctorate
Foundation of Northwestern Polytechnical University (cx201326).
The author is very indebted to Professor E.F. Schmeichel for his
kindness and helpful documents. The author would like to thank Dr.
Jun Ge and Lu Qiao for many helpful discussions, and would like to
express deep gratitude to editors and reviewers for their carefully
reading the earlier version of this work and so many suggestions.


\begin{thebibliography}{10}
\bibitem{Bedrossian}
P.~Bedrossian, \emph{Forbidden Subgraph and Minimum Degree Conditons for
Hamiltonicity}, Ph.D. Thesis, Memphis State University, USA (1991).

\bibitem{Bedrossian_Chen_Schelp}
P. Bedrossian, G. Chen and R.H. Schelp, A generalization of Fan's
condition for Hamiltonicity, pancyclicity, and Hamiltonian connectedness,
\emph{Discrete Math.} {\bf 115} (1993), 39-50.

\bibitem{Benhocine_Wojda}
A. Benhocine and A.P. Wojda, The Geng-Hua Fan conditions for pancyclic
or Hamilton-connected graphs, \emph{J. Combin. Theory Ser. B} {\bf 42}
(1987), 167-180.

\bibitem{Bondy}
J.A. Bondy, Pancyclic graphs I, \emph{J. Combin. Theory Ser. B}
{\bf 11} (1971), 80-84.

\bibitem{Bondy_Murty}
J.A. Bondy and U.S.R. Murty, \emph{Graph Theory with Applications},
Macmillan London and Elsevier, New York (1976).

\bibitem{Fan}
G. Fan, New sufficient conditions for cycles in graphs,
\emph{J. Combin. Theory Ser. B} {\bf 37} (1984), 221-227.

\bibitem{Faudree_Gould}
R.J. Faudree and R.J. Gould, Characterizing forbidden
pairs for hamiltonian properties, \emph{Discrete Math.}
{\bf 173} (1997), 45-60.

\bibitem{Ferrara_Jacobson_Harris}
M. Ferrara, M.S. Jacobson and A. Harris, Cycle lengths in
Hamiltonian graphs with a pair of vertices having large
degree sum, \emph{Graphs and Combin.} {\bf 26} (2010),
215-223.

\bibitem{Ning_Zhang}
B. Ning and S. Zhang, Ore- and Fan-type heavy subgraphs
for Hamiltonicity of 2-connected graphs, \emph{Discrete Math.}
{\bf 313} (2013), 1715-1725.

\bibitem{Schmeichel_Hakimi}
E.F. Schmeichel and S.L. Hakimi, A cycle structure theorem
for hamiltonian graphs, \emph{J. Combin. Theorey Ser. B}
{\bf 45} (1988), 99-107.
\end{thebibliography}
\end{document}